\documentclass{amsart}
\usepackage{amssymb,amsmath,amsfonts,amscd,amsthm,pb-diagram,amsbsy,bm}

\theoremstyle{plain}
\newtheorem{theorem}{Theorem}[section]
\newtheorem*{theorem*}{Main Theorem}
\newtheorem{lemma}{Lemma}[section]

\theoremstyle{definition}

\newcommand{\DP}[1]{(\tilde{\nabla}_{#1}P)}

\numberwithin{equation}{section}

\normalsize

\begin{document}

\title[On hypersurfaces of the nearly K\"ahler $\mathbb{S}^6$ and $\mathbb{S}^3
\times\mathbb{S}^3$] {Hypersurfaces of the homogeneous nearly
K\"ahler $\mathbb{S}^6$ and $\mathbb{S}^3\times\mathbb{S}^3$ with
anticommutative structure tensors}

\author[Z. Hu, Z. Yao and X. Zhang]{Zejun Hu, Zeke Yao and Xi Zhang}



\thanks{This project was supported by NSF of China, Grant Number
11771404.}

\begin{abstract}
Each hypersurface of a nearly K\"ahler manifold is naturally equipped
with two tensor fields of $(1,1)$-type, namely the shape operator
$A$ and the induced almost contact structure $\phi$. In this paper,
we show that, in the homogeneous NK $\mathbb{S}^6$ a hypersurface
satisfies the condition $A\phi+\phi A=0$ if and only if it is totally
geodesic; moreover, similar as for the non-flat complex space forms,
the homogeneous nearly K\"ahler manifold $\mathbb{S}^3\times\mathbb{S}^3$
does not admit a hypersurface that satisfies the condition $A\phi+\phi A=0$.
\end{abstract}

\maketitle

\section{Introduction}\label{sect:1}

The nearly K\"ahler (abbrev. NK) manifold $\mathbb{S}^3 \times
\mathbb{S}^3$ is one of the only four homogeneous $6$-dimensional
nearly K\"ahler spaces (with the remaining three the NK $6$-sphere
$\mathbb{S}^6$, the complex projective space $\mathbb CP^3$ and the
flag manifold $SU(3)/U(1)\times U(1)$, cf. \cite{But1, But2}). Ever
since the groundbreaking research of Bolton-Dillen-Dioos-Vrancken
\cite{B-D-D-V}, people become increasingly interested in the study
of submanifolds of this homogeneous NK $\mathbb{S}^3 \times
\mathbb{S}^3$, and many beautiful results have been established. For
details of the study, besides \cite{B-D-D-V}, we would refer the
readers to \cite{D-L-M-V, H-Z} on almost complex surfaces, to
\cite{B-M-J-L, B-M-J-L-2, D-V-W, H-Z2, Z-D-H-V-W} on Lagrangian
submanifolds, and to \cite{H-Y-Z} on hypersurfaces. It is worth
mentioning that Foscolo and Haskins \cite{F-H} have recently
constructed cohomogeneity one NK structure on both
$\mathbb S^6$ and $\mathbb{S}^3\times\mathbb{S}^3$. Thus, in order
to avoid confusion, from now on in this paper, when we say NK
$\mathbb S^6$ and NK $\mathbb{S}^3 \times \mathbb{S}^3$, we mean
always $\mathbb S^6$ and $\mathbb{S}^3 \times \mathbb{S}^3$ equipped
with the homogeneous NK structures that were elaborate
described in \cite{DOVV} (cf. references therein) and
\cite{B-D-D-V}, respectively.


In the present paper, continuing with our research starting from
\cite{H-Y-Z}, we will focus mainly on hypersurfaces of the NK
$\mathbb{S}^3 \times \mathbb{S}^3$. Recall that given a hypersurface
$M$ of an almost Hermitian manifold with almost complex structure
$J$, it appears on $M$ two naturally defined tensor fields of
$(1,1)$-type: a submanifold structure represented by the shape
operator $A$, and an almost contact structure $\phi$ induced from
$J$. Then, it is an interesting problem to study hypersurfaces with
special relations between $A$ and $\phi$. The first problem one
might study is that the shape operator $A$ and the induced almost
contact structure $\phi$ satisfy the commutativity condition
$A\phi=\phi A$. Indeed, Okumura \cite{Ok} and Montiel-Romero
\cite{M-R} considered real hypersurfaces of the non-flat complex
space forms, and they obtained the classification of such real
hypersurfaces satisfying $A\phi=\phi A$ for complex projective space
\cite{Ok} and complex hyperbolic space \cite{M-R}, respectively.
Moreover, it was shown that hypersurfaces of the homogeneous NK
$\mathbb S^6$ satisfy $A\phi=\phi A$ if and only if they are
geodesic hyperspheres (cf. Theorem 2 of \cite{M} and Remark 2.1 of
\cite{H-Y-Z}). Then following this approach, we have considered a
similar situation for the NK $\mathbb{S}^3\times\mathbb{S}^3$
\cite{H-Y-Z}, our result is the following classification theorem.

\begin{theorem}[cf. \cite{H-Y-Z}]\label{thm:1.1}
Let $M$ be a hypersurface of the homogeneous NK
$\mathbb{S}^3\times\mathbb{S}^3$ that satisfies the condition
$A\phi=\phi A$. Then $M$ is locally given by one of the following
immersions $f_1$, $f_2$ and $f_3$:
\begin{enumerate}
\item[(1)]
$f_1:\ \
\mathbb{S}^3\times\mathbb{S}^2\rightarrow\mathbb{S}^3\times\mathbb{S}^3
\ \ {\rm defined\ by}\ \ (x,y)\mapsto(x,y)$;

\item[(2)]
$f_2:\ \
\mathbb{S}^3\times\mathbb{S}^2\rightarrow\mathbb{S}^3\times\mathbb{S}^3
\ \ {\rm defined\ by}\ \ (x,y)\mapsto(y,x)$;

\item[(3)]
$f_3:\ \
\mathbb{S}^3\times\mathbb{S}^2\rightarrow\mathbb{S}^3\times\mathbb{S}^3
\ \ {\rm defined\ by}\ \ (x,y)\mapsto(\bar{x},y\bar{x})$,
\end{enumerate}
here, $x\in\mathbb{S}^3$, $y\in\mathbb{S}^2$, and as usual
$\mathbb{S}^3$ (resp. $\mathbb{S}^2$) is regarded as the set of the
unit (resp. imaginary) quaternions in the quaternion space
$\mathbb{H}$.
\end{theorem}

One might realize that the next simplest relation between the shape
operator $A$ and the induced almost contact structure $\phi$ is the
anti-commutativity condition $A\phi+\phi A=0$. In this respect, to
our knowledge only Ki-Suh have shown that (cf. Lemma 2.1 and
Proposition 2.2 of \cite{K-S}), by denoting $\bar M^{n}(c)$ the
$n$-dimensional complex space form of constant holomorphic sectional
curvature $c$, if there exists a real hypersurface $M$ of $\bar
M^{n}(c)$ that satisfies the condition $A\phi+\phi A=0$, then $c=0$
and $M$ is cylindrical. To see how about other ambient spaces,
in this paper, we consider the question for two important
$6$-dimensional homogeneous NK manifolds, namely that the
homogeneous NK $\mathbb{S}^6$ and the homogeneous NK
$\mathbb{S}^3\times\mathbb{S}^3$. Our first result is the following
\begin{theorem}\label{thm:1.2}
The totally geodesic hypersurfaces of the homogeneous NK
$\mathbb{S}^6$ are the only hypersurfaces of $\mathbb{S}^6$
satisfying the condition $A\phi+\phi A=0$.
\end{theorem}

For the homogeneous NK $\mathbb{S}^3\times\mathbb{S}^3$, however, in
Theorem 1.1 of \cite{H-Y-Z}, we have shown that it admits neither
totally umbilical hypersurfaces nor hypersurfaces having parallel
second fundamental form. Now, as the second result of this paper, a
further nonexistence theorem can be proved that is stated as below.
\begin{theorem}\label{thm:1.3}
The homogeneous NK $\mathbb{S}^3\times\mathbb{S}^3$ does not admit a
hypersurface that satisfies the condition $A\phi+\phi A=0$.
\end{theorem}

\section{Preliminaries}\label{sect:2}

\subsection{The homogeneous NK structure on $\mathbb{S}^3\times\mathbb{S}^3$}\label{sect:2.1}~

In this subsection, we review some elementary notions and results
from \cite{B-D-D-V}.

By the natural identification $T_{(p,q)}(\mathbb{S}^3 \times
\mathbb{S}^3)\cong T_p\mathbb{S}^3 \oplus T_q\mathbb{S}^3$, we can
write a tangent vector at $(p,q)\in\mathbb{S}^3 \times \mathbb{S}^3$
as $Z(p,q)=(U_{(p,q)},V_{(p,q)})$ or simply $Z=(U,V)$. Then the
well-known almost complex structure $J$ on $\mathbb{S}^3\times
\mathbb{S}^3$ is given by
\begin{equation}\label{eq:2.1}
J Z(p,q) = \tfrac{1}{\sqrt{3}} (2pq^{-1}V - U, -2qp^{-1}U + V).
\end{equation}

Define the Hermitian metric $g$ on $\mathbb{S}^3\times
\mathbb{S}^3$ by
\begin{equation}\label{eq:2.2}
\begin{split}
g(Z, Z')&=\tfrac{1}{2}(\langle Z,Z'\rangle+\langle JZ, JZ' \rangle )\\
&=\tfrac{4}{3}(\langle U,U'\rangle+\langle V,V'\rangle)-\tfrac{2}{3}(\langle p^{-1}U, q^{-1}V'\rangle+\langle p^{-1}U', q^{-1}V\rangle),
\end{split}
\end{equation}
where $Z=(U, V),\ Z'=(U', V')$ are tangent vectors, and
$\langle\cdot,\cdot\rangle$ is the standard product metric on
$\mathbb{S}^3 \times \mathbb{S}^3$. Then $\{g,J\}$ give the
homogeneous NK structure on $\mathbb{S}^3 \times \mathbb{S}^3$.

As usual let $G$ be the (1,2)-tensor field defined by
$G(X,Y):=(\tilde \nabla_X J)Y$, where $\tilde{\nabla}$ is
Levi-Civita connection of $g$. Then, the following further formulas
hold:
\begin{gather}
G(X,Y)+G(Y,X)=0,\label{eq:2.3}\\
G(X,JY)+JG(X,Y)=0,\label{eq:2.4}\\
g(G(X,Y),Z)+g(G(X,Z),Y)=0,\label{eq:2.5}\\
\begin{aligned}\label{eq:2.6}
g(G(X,Y),G(Z,W))=&\tfrac{1}{3}\big[g(X,Z)g(Y,W)-g(X,W)g(Y,Z)\\
&\quad +g(JX,Z)g(JW,Y)-g(JX,W)g(JZ,Y)\big].
\end{aligned}
\end{gather}

An almost product structure $P$ on $\mathbb{S}^3\times\mathbb{S}^3$
is introduced by:
\begin{equation}\label{eq:2.7}
PZ=(pq^{-1}V, qp^{-1}U),\ \
 \forall\, Z=(U,V)\in T_{(p,q)}(\mathbb{S}^3\times \mathbb{S}^3).
\end{equation}
Then we have the following formula for $\tilde\nabla P$:
\begin{equation}\label{eq:2.8}
2\DP{X}Y=JG(X,PY)+JPG(X,Y).
\end{equation}
The curvature tensor $\tilde R$ of the homogeneous NK $\mathbb{S}^3
\times \mathbb{S}^3$ is given by:
\begin{equation}\label{eqn:2.9}
\begin{split}
\tilde{R}(X,Y)Z =& \tfrac{5}{12}\big[g(Y,Z)X - g(X,Z)Y\big]\\
 & +\tfrac{1}{12}\big[g(JY,Z)JX - g(JX,Z)JY - 2g(JX,Y)JZ\big]\\
 & +\tfrac{1}{3}\Big[g(PY,Z)PX - g(PX,Z)PY\\
 & \qquad +g(JPY,Z)JPX - g(JPX,Z)JPY\Big].
\end{split}
\end{equation}

\subsection{Hypersurfaces of the homogeneous NK $\mathbb{S}^3\times\mathbb{S}^3$}\label{sect:2.2}~

Let $M$ be a hypersurface of the homogeneous NK $\mathbb{S}^3
\times\mathbb{S}^3$ with $\xi$ its unit normal vector field. For any
vector field $X$ tangent to $M$, we have the decomposition
\begin{equation}\label{eqn:2.10}
JX=\phi X+ f(X)\xi,
\end{equation}
where $\phi X$ and $f(X)\xi$ are, respectively, the tangent and
normal parts of $JX$. Then $\phi$ is a tensor field of type (1,1),
and $f$ is a $1$-form on $M$. By definition, $\phi$ and $f$ satisfy
the following relations:
\begin{equation}\label{eqn:2.11}
\left\{
\begin{aligned}
&f(X)=g(X,U),\ \ f(\phi X)=0,\ \ \phi^2X=-X+f(X)U,\\
&g(\phi X,Y)=-g(X,\phi Y),\ \ g(\phi X,\phi Y)=g(X,Y)-f(X)f(Y),
\end{aligned}\right.
\end{equation}
where $U:=-J\xi$, which is called the {\it structure vector field}
of $M$. The equations \eqref{eqn:2.11} show that $(\phi,U,f)$ determines
an {\it almost contact structure} over $M$.

Let $\nabla$ be the induced connection on $M$ with $R$ its
Riemannian curvature tensor. The formulas of Gauss and Weingarten
state that
\begin{equation}\label{eqn:2.12}
\begin{split}
\tilde \nabla_X Y=\nabla_X Y + h(X,Y),\quad \tilde \nabla_X \xi=-
A X , \ \ \forall\, X,Y \in TM,
\end{split}
\end{equation}
where $h$ is the second fundamental form, and it is related to the
shape operator $A$ by $h(X,Y)=g(A X,Y)\xi$. Here, using the formulas
of Gauss and Weingarten, we have
\begin{equation}\label{eqn:2.13}
\nabla_X U=\phi AX-G(X,\xi).
\end{equation}

The Gauss and Codazzi equations of $M$ are given by
\begin{equation}\label{eqn:2.14}
\begin{split}
R(X,Y)Z=&\tfrac{5}{12}\big[g(Y,Z)X - g(X,Z)Y\big]\\
 &+ \tfrac{1}{12}\big[g(\phi Y,Z)\phi X - g(\phi X,Z)\phi Y - 2g(\phi X,Y)\phi Z\big]\\
 &+ \tfrac{1}{3}\Big[g(PY,Z)(PX)^\top - g(PX,Z)(PY)^\top\\
 &\qquad\,\ +g(JPY,Z)(JPX)^\top - g(JPX,Z)(JPY)^\top\Big]\\&
 +g(AZ,Y)AX-g(AZ,X)AY,
 \end{split}
\end{equation}
\begin{equation}\label{eqn:2.15}
\begin{split}
(\nabla_X A)Y-(\nabla_Y A)X=&\tfrac{1}{12}\big[g(X,U)\phi Y-g(Y,U)\phi X-2g(\phi X,Y)U\big]\\
&+ \tfrac{1}{3}\Big[g(PX,\xi)(PY)^\top-g(PY,\xi)(PX)^\top\\
 &\qquad+g(PX,U)(JPY)^\top-g(PY,U)(JPX)^\top\Big],
\end{split}
\end{equation}
where $\cdot^\top$ denotes the tangential part.

Following the usual terminology, we call a hypersurface $M$ of the
NK $\mathbb{S}^3 \times \mathbb{S}^3$ the {\it Hopf hypersurface} if
the integral curves of the structure vector field $U$ are geodesics
of $M$, that is $\nabla_U U=0$. It is also equivalent that the
structure vector field $U$ is a principal direction, with principal
curvature function denoted by $\mu$. A basic lemma for Hopf
hypersurfaces of the NK $\mathbb{S}^3 \times \mathbb{S}^3$ is
stated as follows:
\begin{lemma}\label{lemma:2.1}
Let $M$ be a Hopf hypersurface in the homogeneous NK $\mathbb{S}^3
\times \mathbb{S}^3$. Then we have
\begin{equation}
\begin{aligned}\label{eq:2.16}
\tfrac16&g(\phi X,Y)-\tfrac23\big[g(PX,\xi)g(PY,U)-g(PX,U)g(PY,\xi)\big]\\
&=g((\mu I-A)G(X,\xi),Y)+g(G((\mu I-A)X,\xi),Y)\\
&\ \ \ -\mu g((A\phi +\phi A)X,Y)+2g(A\phi AX,Y),\ \ X,Y\in
\{U\}^{\bot},
\end{aligned}
\end{equation}
where $\{U\}^{\bot}$ denotes a distribution of $TM$ that is
orthogonal to $U$, and $I$ denotes the identity transformation.
\end{lemma}

\begin{proof}
A direct calculation of $(\nabla_{X}A)U$, with using $AU=\mu U$,
\eqref{eqn:2.13}, we have
\begin{equation}\label{eq:2.17}
(\nabla_{X}A)U=X(\mu)U+(\mu I-A)(-G(X,\xi)+\phi AX).
\end{equation}

It follows that, for $\forall\, X,Y\in \{U\}^{\bot}$,
\begin{equation}\label{eq:2.18}
g((\nabla_{X}A)Y,U)=g((\nabla_{X}A)U,Y)=g((\mu I-A)(-G(X,\xi)+\phi
AX),Y).
\end{equation}

Thus, we have
\begin{equation}
\begin{aligned}\label{eq:2.19}
g((\nabla_{X}A)Y-(\nabla_{Y}A)X,U)=&-g((\mu I-A)G(X,\xi),Y)-2g(A\phi AX,Y)\\
&-g(G((\mu I-A)X,\xi),Y)+\mu g((A\phi+\phi A)X,Y).
\end{aligned}
\end{equation}

On the other hand, by using the Codazzi equation \eqref{eqn:2.15},
we get
\begin{equation}\label{eq:2.20}
\begin{aligned}
&g((\nabla_{X}A)Y-(\nabla_{Y}A)X,U)\\
&=-\tfrac{1}{6}g(\phi
X,Y)+\tfrac{2}{3}(g(PX,\xi)g(PY,U)-g(PX,U)g(PY,\xi)).
\end{aligned}
\end{equation}

From \eqref{eq:2.19} and \eqref{eq:2.20}, we immediately get
\eqref{eq:2.16}.
\end{proof}

Before concluding this section, following that in \cite{H-Y-Z} we
introduce the distribution $\mathfrak{D}$. When we study hypersurfaces of
the NK $\mathbb{S}^3\times\mathbb{S}^3$, the consideration of $\mathfrak{D}$
is very helpful for the choice of a canonical frame. Precisely, for each point
$p\in M$, we define
$$
\mathfrak{D}(p):={\rm Span}\,\{\xi(p),U(p),P\xi(p),PU(p)\}.
$$

Since $P$ is anti-commutative with $J$, it is clear that $\mathfrak{D}$
defines a distribution on $M$ with dimension $2$ or $4$, and that it is
invariant under the action of both $J$ and $P$. Along $M$, let
$\mathfrak{D}^\bot$ denote the distribution in $T(\mathbb{S}^3
\times \mathbb{S}^3)$ that is orthogonal to $\mathfrak{D}$ at each
$p\in M$.

If $\dim\mathfrak{D}=4$ holds in an open set, then there exists a
unit tangent vector field $e_1\in \mathfrak{D}$ and functions
$a,b,c$ with $c>0$ such that
\begin{equation}\label{eqn:2.21}
P\xi=a\xi+bU+ce_1,\ \ a^2+b^2+c^2=1.
\end{equation}

Put $e_2=Je_1$. From the fact $\dim\,\mathfrak{D}^\bot=2$ and that
$\mathfrak{D}^\bot$ is invariant under the action of both $J$ and
$P$, we can choose a local unit vector field
$e_3\in\mathfrak{D}^\bot$ such that $Pe_3=e_3$. Put $e_4=Je_3$ and
$e_5=U$. Then $\{e_i\}_{i=1}^5$ is a well-defined orthonormal basis
of $TM$ and, acting by $P$, it has the following properties:
\begin{equation}\label{eqn:2.22}
\left\{
\begin{aligned}
&P\xi=a\xi+ce_1+be_5,\ \ Pe_1=c\xi-ae_1-be_2,\\
&Pe_2=ce_5-be_1+ae_2,\ \ Pe_3=e_3,\\
&Pe_4=-e_4,\ \ Pe_5=b\xi+ce_2-ae_5.
\end{aligned}\right.
\end{equation}

If $\dim\mathfrak{D}=2$ holds in an open set, then we can write
\begin{equation}\label{eqn:2.23}
P\xi=a\xi+bU,\ \ a^2+b^2=1.
\end{equation}

Now, $\mathfrak{D}^\bot$ is a $4$-dimensional distribution that is
invariant under the action of both $J$ and $P$. Hence, we can choose
unit vector fields $e_1,\, e_3\in\mathfrak{D}^\bot$ such that
$Pe_1=e_1,\, Pe_3=e_3$. Put $e_2=Je_1,\,e_4=Je_3$ and $e_5=U$. In
this way, we obtain an orthonormal basis $\{e_i\}_{i=1}^5$ of $TM$.
However, we would remark that such choice of $\{e_1, e_3\}$ (resp.
$\{e_2, e_4\}$) is unique up to an orthogonal transformation.

\section{Proof of Theorem \ref{thm:1.2}}\label{sect:3}

For basic results of the well-known NK $\mathbb{S}^6$, i.e., the
six-dimensional unit sphere $\mathbb{S}^6$ equipped with a
homogeneous NK structure $(J,g)$, of which $J$ is the almost complex
structure defined by using the vector cross product of purely
imaginary Cayley numbers $\mathbb{R}^7$ and $g$ is the metric
induced from the Euclidean space $\mathbb{R}^7$, we refer to
\cite{DOVV} and the references therein.

Let $M$ be an orientable hypersurface of the NK $\mathbb{S}^6$ with
$\xi$ its unit normal vector field. Then, the equations from
\eqref{eqn:2.10} up to \eqref{eqn:2.13} in subsection \ref{sect:2.2}
also hold, so that $M$ admits an almost contact metric structure
$(\phi,U,f,g)$ induced from the NK structure of $\mathbb{S}^6$,
whereas the Codazzi equation becomes
\begin{equation}\label{eqn:3.1}
(\nabla_X A)Y=(\nabla_Y A)X,\ \ \forall\, X,Y\in TM.
\end{equation}

For the NK $\mathbb{S}^6$, totally geodesic hypersurfaces do exist
and they trivially satisfy the relation $A\phi+\phi A=0$.

Now, we assume that $M$ is an orientable hypersurface of the NK
$\mathbb{S}^6$ that satisfies the condition $A\phi+\phi A=0$. Then,
by definition $\phi U=0$, we have $AU=\mu U$, i.e., $M$ is a
Hopf hypersurface and, $\mu$ is the principal curvature function
corresponding to the structure vector field $U$. Moreover, if
$X\in\{U\}^\bot$ is a principal vector field with principal
curvature function $\lambda$, then $A\phi X=-\phi AX=-\lambda \phi
X$ implies that $\phi X$ is also a principal vector field with
principal curvature function $-\lambda$.

Recall that Berndt-Bolton-Woodward (Theorem 2 of
\cite{B-B-W}) proved that a connected Hopf hypersurface of the NK
$\mathbb{S}^6$ is an open part of either a geodesic hypersphere of
$\mathbb{S}^6$ or a tube around an almost complex curve in the NK
$\mathbb{S}^6$, and the principal curvature function $\mu$ is
constant (Lemma 2 of \cite{B-B-W}).

Similar to the proof of Lemma \ref{lemma:2.1}, for Hopf
hypersurfaces of the NK $\mathbb{S}^6$, we can easily show that, by
using \eqref{eqn:2.13}, the following basic equation holds:
\begin{equation}\label{eqn:3.2}
\begin{aligned}
g((\mu I-A)&G(X,\xi),Y)+g(G((\mu I-A)X,\xi),Y)\\
&-\mu g((A\phi +\phi A)X,Y)+2g(A\phi AX,Y)=0,\ \ X,Y\in TM.
\end{aligned}
\end{equation}

If $M$ is a geodesic hypersphere, then $M$ is totally
umbilical and we have a function $\lambda$ on $M$ such that
$AX=\lambda X, \forall\, X\in TM$. This together with $A\phi+\phi
A=0$ implies that $\lambda=0$. Hence, $M$ is a totally geodesic
hypersurface.

If $M$ is a tube around an almost complex curve $\Gamma$ with radius
$r$ in $\mathbb{S}^6$, then, according to the proof of Proposition 2
and subsequent Remark in \cite{B-B-W}, we have $AU=-\cot r\,U$, and
the remaining principal curvatures on the distribution
$\{U\}^{\bot}$ are $\tan(\theta+ r),\,\tan(\theta-r)$ and $-\cot r$
for $\theta\in[0,\frac{\pi}{2})$ which is a function on $M$. Moreover, as
\cite{B-B-W} has pointed out, the hypersurface $M$ has exactly two
or three distinct principal curvatures at each point. We denote by
$\nu,\ 2\leq\nu\leq3,$ the maximum number of distinct principal
curvatures on $M$, then the set $M_\nu=\{x\in M|$$M$ has exactly
$\nu$ distinct principal curvatures at $x$\} is a non-empty open
subset of $M$. By the continuity of the principal curvature
function, each connected component of $M_\nu$ is an open subset, and
the multiplicities of distinct principal curvatures remain
unchanged on each connected component of $M_\nu$, so we can find a
local smooth frame field with respect to the principal curvatures.
The following discussion will be divided into two cases, depending on
the value of $\nu$.

\vskip 1mm

{\bf Case I}. $\nu=3$.

In this case, on each connected component of $M_3$, the
multiplicities of the distinct principal curvatures, namely
$\tan(\theta+r)$, $\tan(\theta-r)$ and $-\cot r$, should be $1,1$
and $3$, respectively. Then we have an orthonormal frame field
$\{X_i\}_{i=1}^5$ such that
\begin{equation}\nonumber
\left\{
\begin{aligned}
&AX_1=\tan(\theta+r)X_1,\ AX_2=\tan(\theta-r)X_2,\ AX_3=-\cot rX_3,\\
&AX_4=-\cot rX_4,\ AX_5=-\cot rX_5,\ X_5=U.
\end{aligned}\right.
\end{equation}

Applying the condition $A\phi+\phi A=0$, we have
$$
A\phi X_1=-\tan(\theta+r)\phi X_1,\ \ A\phi X_2=-\tan(\theta-r)\phi
X_2,\ \ A\phi X_3=\cot r\phi X_3.
$$

Taking $X=X_1$ and $Y=\phi X_1$ in \eqref{eqn:3.2}, and using
$A\phi+\phi A=0$, we get $\tan(\theta+r)=0$.
Analogously, taking $X=X_2$ and $Y=\phi X_2$ in \eqref{eqn:3.2}, we
get $\tan(\theta-r)=0$, which is a contradiction with
$\tan(\theta+r)\neq\tan(\theta-r)$. Thus, {\bf
Case I} does not occur.

\vskip 1mm

{\bf Case II}. $\nu=2$.

In this case, $M$ has exactly two distinct principal curvatures,
that is, two of the three principal curvatures $\tan(\theta+r)$,
$\tan(\theta-r)$ and $-\cot r$ are equal. Without loss of
generality, we assume that $\tan(\theta+r)=-\cot r$, so that the
multiplicities of the distinct principal curvatures, namely
$\tan(\theta-r)$ and $-\cot r$, are $1$ and $4$, respectively. Then,
we have an orthonormal frame field $\{X_i\}_{i=1}^5$ such that
\begin{equation}\nonumber
\left\{
\begin{aligned}
&AX_1=\tan(\theta-r)X_1,\ AX_2=-\cot rX_2,\ AX_3=-\cot rX_3,\\
&AX_4=-\cot rX_4,\ AX_5=-\cot rX_5,\ X_5=U.
\end{aligned}\right.
\end{equation}

Applying $A\phi+\phi A=0$, we get $A\phi X_1=-\tan(\theta-r)\phi
X_1$ and $A\phi X_2=\cot r\phi X_2$.
Then taking in \eqref{eqn:3.2} that $(X,Y)=(X_1,\phi X_1)$ and $(X,Y)=(X_2,\phi X_2)$,
respectively, we immediately get
$\tan(\theta-r)=-\cot r=0$. This is again a contradiction.

This completes the proof of Theorem \ref{thm:1.2}.\qed

\section{Proof of Theorem \ref{thm:1.3}}\label{sect:4}

To give the proof, we assume that $M$ is a hypersurface of the NK
$\mathbb{S}^3 \times\mathbb{S}^3$ which satisfies the condition
$A\phi+\phi A=0$. Then, by the fact $\phi U=0$, we see that $M$ is a
Hopf hypersurface with $AU=\mu U$. Moreover, if $X\in\{U\}^\bot$ is
a principal vector field with principal curvature function
$\lambda$, i.e., $AX=\lambda X$, then $A\phi X=-\phi AX=-\lambda
\phi X$ implies that $\phi X$ is also a principal vector field with
principal curvature function $-\lambda$. We denote
$\lambda,\,-\lambda,\,\beta,\,-\beta$ with $\lambda\geq0$ and
$\beta\geq0$ the four principal curvatures on distribution
$\{U\}^{\bot}$.
Since the only possible dimension of $\mathfrak{D}$ is $2$ or $4$,
we will divide the proof of Theorem \ref{thm:1.3} into the proofs of
two Lemmas. First, we have the following Lemma.

\begin{lemma}\label{lemma:4.1}
The case $\dim\mathfrak{D}=4$ does not occur.
\end{lemma}
\begin{proof}
Suppose that $\dim\mathfrak{D}=4$ does occur on some point of $M$.
We denote by $\Omega=\{x\in M|$ the dimension of $\mathfrak{D}$ is 4 at $x$\}, then $\Omega$ is an open set of $M$.
Since $A\phi+\phi A=0$, we can write \eqref{eq:2.16} on $\Omega$ as
\begin{equation}
\begin{aligned}\label{eq:4.1}
\tfrac16g(\phi X,Y)&-\tfrac23\big[g(PX,\xi)g(PY,U)-g(PX,U)g(PY,\xi)\big]=-2g(\phi A^2X,Y)\\
&+g\big((\mu I-A)G(X,\xi),Y\big)+g\big(G((\mu I-A)X,\xi),Y\big),\ \ \ \ X,Y\in
\{U\}^{\bot}.
\end{aligned}
\end{equation}

We denote by $\nu$ $(\nu\leq5)$ the maximum number on $\Omega$ of
distinct principal curvatures, then the set $\Omega_\nu:=\{x\in
\Omega\,|\,M$ has exactly $\nu$ distinct principal curvatures at
$x$\} is a non-empty open subset of $M$. By the continuity of the
principal curvature function, each connected component of
$\Omega_\nu$ is an open subset, the multiplicities of distinct
principal curvatures remain unchanged on each connected component
of $\Omega_\nu$, so we can find a local smooth frame field with
respect to the principal curvatures. From Theorem 1.1 of
\cite{H-Y-Z}, we know that $M$ can not be totally umbilical, even
locally. So the following discussion will be divided into four
cases, depending on the value of $\nu,\ 2\leq\nu\leq5$.

\vskip 1mm

\noindent{\bf Case I}. $\nu=5$.

\vskip 1mm

In this case, on each connected component of $\Omega_5$, we can have an orthonormal frame field
$\{X_{i}\}_{i=1}^{5}$  such that
\begin{equation}\label{eq:4.2}
AX_{1}=\lambda X_{1},\ \ AX_{2}=\beta X_{2},\ \ AX_{3}=-\lambda X_{3},
\ \ AX_{4}=-\beta X_{4},\ \ AX_{5}=\mu X_{5},
\end{equation}
where $X_{3}=JX_{1},\, X_{4}=JX_{2},\, X_{5}=U$. As $\nu=5$, we have
$\lambda>0,\, \beta>0,\,\lambda\neq\beta$ and
$\mu\not\in\{\lambda,-\lambda,\,\beta,-\beta\}$. Let
$\{e_{i}\}_{i=1}^{5}$ be the frame field as described in
\eqref{eqn:2.22}. Then, by assuming that $X_{i}=\sum_{j=1} ^4
a_{ij}e_{j}$ for $1\leq i\leq 4$, we have $(a_{ij})\in SO(4)$, and
by the choice of $\{e_{i}\}_{i=1}^{5}$ it holds that
\begin{equation}\label{eq:4.3}
a_{i+2,j}=(-1)^{j}a_{i,3-j},\ \ a_{i+2,j+2}=(-1)^{j}a_{i,5-j}, \ \
i,j=1,2.
\end{equation}

First, taking $X=X_i$ and $Y=X_j$ in \eqref{eq:4.1} for $1\leq
i<j\leq4$, using \eqref{eq:2.3}--\eqref{eq:2.5} and \eqref{eqn:2.22},
we can derive the following equations:
\begin{equation}\label{eq:4.4}
-\tfrac{1}{6}+\tfrac{2}{3}c^{2}a_{11}^{2}+\tfrac{2}{3}c^{2}a_{12}^{2}=2\lambda^{2},
\end{equation}
\begin{equation}\label{eq:4.5}
-\tfrac{1}{6}+\tfrac{2}{3}c^{2}a_{21}^{2}+\tfrac{2}{3}c^{2}a_{22}^{2}=2\beta^{2},
\end{equation}
\begin{equation}\label{eq:4.6}
\tfrac{2}{3}c^{2}a_{11}a_{21}+\tfrac{2}{3}c^{2}a_{12}a_{22}=(2\mu+\lambda-\beta)g(G(X_{1},X_{2}),U),
\end{equation}
\begin{equation}\label{eq:4.7}
\tfrac{2}{3}c^{2}a_{11}a_{21}+\tfrac{2}{3}c^{2}a_{12}a_{22}=-(2\mu-\lambda+\beta)g(G(X_{1},X_{2}),U),
\end{equation}
\begin{equation}\label{eq:4.8}
\tfrac{2}{3}c^{2}a_{11}a_{22}-\tfrac{2}{3}c^{2}a_{12}a_{21}=(2\mu-\lambda-\beta)g(G(X_{1},X_{2}),\xi),
\end{equation}
\begin{equation}\label{eq:4.9}
\tfrac{2}{3}c^{2}a_{11}a_{22}-\tfrac{2}{3}c^{2}a_{12}a_{21}=-(2\mu+\lambda+\beta)g(G(X_{1},X_{2}),\xi).
\end{equation}

The equations \eqref{eq:4.6} and \eqref{eq:4.7}, \eqref{eq:4.8}
and \eqref{eq:4.9} imply that
\begin{equation}\label{eq:4.10}
4\mu g(G(X_{1},X_{2}),U)=0,\ \ 4\mu g(G(X_{1},X_{2}),\xi)=0.
\end{equation}

From \eqref{eq:2.3}, \eqref{eq:2.4} and \eqref{eq:2.5} we see that,
for $1\leq i\leq4$, it holds $g(G(X_{1},X_{2}),X_{i})=0$. Thus,
$G(X_{1},X_{2})\in {\rm Span}\,\{\xi,U\}$. On the other hand, from
\eqref{eq:2.6}, we have
\begin{equation}\label{eq:4.11}
g(G(X_{1},X_{2}),G(X_{1},X_{2}))=\tfrac{1}{3}.
\end{equation}
It follows from \eqref{eq:4.10} that $\mu=0$.

Second, from the fact $AU=0$, we have
\begin{equation}\label{eq:4.12}
(\nabla_XA)U-(\nabla_{U}A)X=-A\nabla_XU-\nabla_{U}AX+A\nabla_{U}X.
\end{equation}

On the other hand, applying \eqref{eqn:2.22} to the Codazzi equation
\eqref{eqn:2.15}, we can get
\begin{equation}\label{eq:4.13}
(\nabla_{e_{1}}A)U-(\nabla_{U}A)e_{1}=-\tfrac{1}{12}e_{2}-\tfrac13\big[2acU-2abe_{1}+(2a^{2}-1)e_{2}\big],
\end{equation}
\begin{equation}\label{eq:4.14}
(\nabla_{e_2}A)U-(\nabla_{U}A)e_2=\tfrac{1}{12}e_1-\tfrac13\big[2bcU+(1-2b^2)e_1+2abe_2\big].
\end{equation}

Then, from \eqref{eq:4.12} and \eqref{eq:4.13}, calculating the
$U$-component of both the right hand sides, we can get $ac=0$.
Analogously, from \eqref{eq:4.12} and \eqref{eq:4.14}, we can get
$bc=0$. Therefore, according to \eqref{eqn:2.21}, we have $a=b=0$
and $c=1$.

Third, in order to apply the Codazzi equations, we need to calculate the
connections $\{\nabla_{X_{i}}X_{j}\}$. Put $\nabla_{X_i}X_j=\sum \Gamma_{ij}^{k}X_{k}$ with
$\Gamma_{ij}^{k}=-\Gamma_{ik}^{j}$, $1\le i,j,k\le 5$. Assume that
\begin{equation}\label{eq:4.15}
g(G(X_{1},X_{2}),\xi)=k,\ \ g(G(X_{1},X_{2}),U)=l.
\end{equation}
Then \eqref{eq:4.11} and the fact $G(X_{1},X_{2})\in {\rm
Span}\,\{\xi,U\}$ show that $k^{2}+l^{2}=\frac{1}{3}$.

By definition and the Gauss and Weingarten formulas, we have the
calculation
\begin{equation*}
G(X_1,\xi)=-\sum_{i=1}^5 \Gamma_{15}^{i}X_i+\lambda X_3.
\end{equation*}
However, according to \eqref{eq:4.15}, we also have
$G(X_1,\xi)=-kX_2+lX_4$. Hence, we obtain
\begin{equation}\label{eq:4.16}
\Gamma_{15}^{1}=0,\ \ \Gamma_{15}^{2}=k,\ \ \Gamma_{15}^{3}=\lambda,\ \ \Gamma_{15}^{4}=-l.
\end{equation}

Similarly, taking $(X,Y)=(X_i,\xi)$ in $G(X,Y)=(\tilde \nabla_X J)Y$
for $2\leq i\leq 4$, and by use of \eqref{eq:4.15}, we further obtain
\begin{equation}\label{eq:4.17}
\left\{
\begin{aligned}
&\Gamma_{25}^{1}=-k,\ \ \Gamma_{25}^{2}=0,\ \ \Gamma_{25}^{3}=l,\ \ \Gamma_{25}^{4}=\beta,\\
&\Gamma_{35}^{1}=\lambda,\ \ \Gamma_{35}^{2}=-l,\ \ \Gamma_{35}^{3}=0,\ \ \Gamma_{35}^{4}=-k,\\
&\Gamma_{45}^{1}=l,\ \ \Gamma_{45}^{2}=\beta,\ \ \Gamma_{45}^{3}=k,\ \ \Gamma_{45}^{4}=0.
\end{aligned}\right.
\end{equation}

Moreover, by using \eqref{eq:4.15} and the Gauss and Weingarten formulas, we get
\begin{equation}\label{eq:4.18}
lX_2+kX_4=G(U,X_1)=\sum_{i=1}^5\Gamma_{53}^{i}X_i-\sum_{i=1}^5 \Gamma_{51}^{i}JX_i.
\end{equation}

It follows that
\begin{equation}\label{eq:4.19}
\Gamma_{53}^{2}+\Gamma_{51}^{4}=l,\ \ \Gamma_{53}^{4}-\Gamma_{51}^{2}=k.
\end{equation}

Finally, we will calculate the expressions $(\nabla_U
A){e_i}-(\nabla_{e_i} A)U$ for $1\leq i\leq 4$.

On one hand, for each $1\leq i\leq 4$, we directly calculate
$(\nabla_U A)e_{i}-(\nabla_{e_{i}} A)U$, with the use of
$e_{i}=\sum_{j=1}^4{a_{ji}}X_{j}$ and the preceding results
\eqref{eq:4.16} and \eqref{eq:4.17}. Then we get an expression for
$(\nabla_U A){e_i}-(\nabla_{e_i} A)U$ in terms of the frame field
$\{X_{i}\}_{i=1}^{4}$.

On the other hand, for each $1\leq i\leq 4$, we calculate
$(\nabla_U A)e_{i}-(\nabla_{e_{i}} A)U$ by the Codazzi equation
\eqref{eqn:2.15}. Then, by using \eqref{eqn:2.22} and
$e_{i}=\sum_{j=1}^4{a_{ji}}X_{j}$, we get another expression of
$(\nabla_UA){e_i}-(\nabla_{e_i} A)U$ in terms of the frame field
$\{X_{i}\}_{i=1}^{4}$.

In this way, comparing both calculations of $(\nabla_U
A)e_{i}-(\nabla_{e_{i}} A)U$ for each $1\leq i\leq 4$, we get a
matrices equation $C=(a_{ij})^TB$, where
$$
C=
\left(
  \begin{array}{cccc}
    -\tfrac{1}{4}a_{12} & -\tfrac{1}{4}a_{22} & -\tfrac{1}{4}a_{11} & -\tfrac{1}{4}a_{21} \\
    \tfrac{1}{4}a_{11} & \tfrac{1}{4}a_{21} & -\tfrac{1}{4}a_{12} & -\tfrac{1}{4}a_{22} \\
    \tfrac{1}{12}a_{14} & \tfrac{1}{12}a_{24} & \tfrac{1}{12}a_{13} & \tfrac{1}{12}a_{23} \\
    -\tfrac{1}{12}a_{13} & -\tfrac{1}{12}a_{23} & \tfrac{1}{12}a_{14} & \tfrac{1}{12}a_{24}
  \end{array}
\right),
$$
$$
B=
\left(
  \begin{array}{cccc}
    U(\lambda) & (\lambda-\beta)\Gamma_{51}^{2}+\beta k & 2\lambda \Gamma_{51}^{3}-\lambda^{2} & (\lambda+\beta)\Gamma_{51}^{4}+\beta l \\
    (\beta-\lambda)\Gamma_{52}^{1}-\lambda k & U(\beta) & (\lambda+\beta)\Gamma_{52}^{3}-\lambda l & 2\beta \Gamma_{52}^{4}-\beta^{2} \\
    -2\lambda \Gamma_{53}^{1}+\lambda^{2} & (-\lambda-\beta)\Gamma_{53}^{2}-\beta l & -U(\lambda) & (\beta-\lambda)\Gamma_{53}^{4}+\beta k \\
    (-\lambda-\beta)\Gamma_{54}^{1}+\lambda l & -2\beta \Gamma_{54}^{2}+\beta^{2} & (\lambda-\beta)\Gamma_{54}^{3}-\lambda k & -U(\beta) \\
  \end{array}
\right).
$$
Thus, $B=(a_{ij})C:=(B_{ij})$. Using \eqref{eq:4.3}, it is
straightforward to verify that $B=(a_{ij})C$ is skew-symmetric. From
the facts $B_{12}+B_{21}=0$ and $\lambda\neq\beta$, we have
$\Gamma_{51}^{2}=\frac{k}{2}$. Moreover, from the facts
$B_{34}+B_{43}=0$ and $\lambda\neq\beta$, we have
$\Gamma_{53}^{4}=-\frac{k}{2}$. Combining these with \eqref{eq:4.19}
we get $k=0$. Analogously, from the facts $B_{23}+B_{32}=0$,
$B_{14}+B_{41}=0$, $\lambda+\beta\neq0$ and \eqref{eq:4.19}, we can
further get $l=0$. Thus, we get a contradiction to
$k^{2}+l^{2}=\tfrac{1}{3}$. This implies that {\bf Case I} does not
occur.

\vskip 1mm

\noindent{\bf Case II}. $\nu=4$.

\vskip 1mm

In this case, on a connected component of $\Omega_4$, without loss of generality, we are sufficient to
consider the following two subcases:

\noindent{\bf II-(i)}: $\lambda\neq \beta,\ \lambda>0,\ \beta>0$ and
$\mu\in\{\lambda,\beta,-\lambda,-\beta\}$.

\noindent{\bf II-(ii)}: $\lambda=0,\ \beta>0$ and
$\mu\not\in\{0,\beta,-\beta\}$.

For both of the above two subcases, following similar arguments as
the discussion of Case I from \eqref{eq:4.2} up to
\eqref{eq:4.11}, we can also get $\mu=0$. This is a
contradiction, showing that {\bf Case II} does not occur.

\vskip 1mm

\noindent{\bf Case III}. $\nu=3$.

\vskip 1mm

In this case, on a connected component of $\Omega_3$, without loss of generality, we are sufficient to
consider the following three subcases:

\noindent{\bf III-(i)}: $\lambda=0,\ \beta>0$ and $\mu\in\{\beta,-\beta\}$.

\noindent{\bf III-(ii)}: $\lambda=\mu=0$ and $\beta>0$.

\noindent{\bf III-(iii)}: $\lambda=\beta>0$ and
$\mu\not\in\{\lambda,-\lambda\}$.

\vskip 1mm

In case {\bf III-(i)}, similar arguments as
the discussion of Case I from \eqref{eq:4.2} up to
\eqref{eq:4.11}, we can get $\mu=0$. Thus, we get a contradiction.

\vskip 1mm

In case {\bf III-(ii)}, taking an orthonormal frame field
$\{X_{i}\}_{i=1}^{5}$ satisfying \eqref{eq:4.2}, we still have the
equations from \eqref{eq:4.4} up to \eqref{eq:4.14}. Then we can get
$c=1$. By calculating \eqref{eq:4.4}+\eqref{eq:4.5} and that $(a_{ij})\in
SO(4)$, we further have the conclusion
\begin{equation}\label{eq:4.20}
\lambda^{2}+\beta^{2}=\tfrac{1}{6}.
\end{equation}
By $\lambda=0$, we have $\beta=\frac{\sqrt{6}}{6}$. Then \eqref{eq:4.4}
and \eqref{eq:4.5} give that
\begin{equation}\label{eq:4.21}
a_{11}^{2}+a_{12}^{2}=\tfrac{1}{4},\ \
a_{21}^{2}+a_{22}^{2}=\tfrac{3}{4}.
\end{equation}

On the other hand, making the summation $\eqref{eq:4.6}^2+\eqref{eq:4.8}^2$,
we easily see that
\begin{equation*}
(a_{11}^{2}+a_{12}^{2})(a_{21}^{2}+a_{22}^{2})=\tfrac{1}{8},
\end{equation*}
which is a contradiction to \eqref{eq:4.21}.

\vskip 1mm

In case {\bf III-(iii)}, taking an orthonormal frame field
$\{X_{i}\}_{i=1}^{5}$ satisfying \eqref{eq:4.2}, we can also
derive the equations from \eqref{eq:4.4} up to \eqref{eq:4.11},
thus we have $\mu=0$. Then, similarly, we have the equations from
\eqref{eq:4.12} up to \eqref{eq:4.14}, so we get in \eqref{eqn:2.22}
that $a=b=0$ and $c=1$, and by calculating
\eqref{eq:4.4}+\eqref{eq:4.5}, we get
$\lambda=\beta=\tfrac{\sqrt{3}}{6}$. It follows from \eqref{eq:4.4},
\eqref{eq:4.5} and \eqref{eq:4.6} that
\begin{equation}\label{eq:4.22}
a_{11}^{2}+a_{12}^{2}=\tfrac{1}{2},\ \
a_{21}^{2}+a_{22}^{2}=\tfrac{1}{2},\ \ a_{11}a_{21}+a_{12}a_{22}=0.
\end{equation}

Let us put $a_{11}=\tfrac{1}{\sqrt{2}}\cos \theta_{1}$,
$a_{12}=\tfrac{1}{\sqrt{2}}\sin \theta_{1}$,
$a_{21}=\tfrac{1}{\sqrt{2}}\cos \theta_{2}$ and
$a_{22}=\tfrac{1}{\sqrt{2}}\sin \theta_{2}$. Then
$0=a_{11}a_{21}+a_{12}a_{22}=\tfrac{1}{2}\cos(\theta_{1}-\theta_{2})$
implies that $\theta_{1}-\theta_{2}=\tfrac{\pi}{2}(2k+1)$,
$k\in\mathbb Z$. Therefore, we have either $(a_{21},a_{22})=(a_{12},-a_{11})$
or $(a_{21},a_{22})=(-a_{12},a_{11})$. If
necessary by taking $-X_2$ instead of $X_2$, we are sufficient to
consider the case that $a_{21}=a_{12}$ and $a_{22}=-a_{11}$.

From \eqref{eq:4.22} and that $(a_{ij})\in SO(4)$, we further have
\begin{equation*}
a_{13}^{2}+a_{14}^{2}=\tfrac{1}{2},\ \
a_{23}^{2}+a_{24}^{2}=\tfrac{1}{2},\ \ a_{13}a_{23}+a_{14}a_{24}=0.
\end{equation*}
This implies that, similar to the preceding paragraph,
$(a_{23},a_{24})=(a_{14},-a_{13})$ or $(a_{23},a_{24})$ $=(-a_{14},a_{13})$.
If $a_{23}=a_{14}$ and $a_{24}=-a_{13}$, then
$X_{2}=-X_{3}$, which is impossible. Thus, $a_{23}=-a_{14}$ and
$a_{24}=a_{13}$ hold.

For simplicity, we put $m=-\tfrac{2\sqrt{6}}{3}a_{13}a_{14}$ and
$n=\tfrac{\sqrt{6}}{3}(a_{14}^2-a_{13}^2)$. Then
$m^{2}+n^{2}=\tfrac{1}{6}$.

Now, from \eqref{eqn:2.22} we can express $\{PX_i\}_{i=1}^4$ as
follows:
\begin{equation}\label{eqn:4.23}
\left\{
\begin{aligned}
&PX_1=a_{11}\xi+a_{12}U-\tfrac{\sqrt{6}}{2}nX_1+\tfrac{\sqrt{6}}{2}mX_2+\tfrac{\sqrt{6}}{2}mX_3+\tfrac{\sqrt{6}}{2}nX_4,\\
&PX_2=a_{12}\xi-a_{11}U+\tfrac{\sqrt{6}}{2}mX_1+\tfrac{\sqrt{6}}{2}nX_2+\tfrac{\sqrt{6}}{2}nX_3-\tfrac{\sqrt{6}}{2}mX_4,\\
&PX_3=-a_{12}\xi+a_{11}U+\tfrac{\sqrt{6}}{2}mX_1+\tfrac{\sqrt{6}}{2}nX_2+\tfrac{\sqrt{6}}{2}nX_3-\tfrac{\sqrt{6}}{2}mX_4,\\
&PX_4=a_{11}\xi+a_{12}U+\tfrac{\sqrt{6}}{2}nX_1-\tfrac{\sqrt{6}}{2}mX_2-\tfrac{\sqrt{6}}{2}mX_3-\tfrac{\sqrt{6}}{2}nX_4.
\end{aligned}\right.
\end{equation}

Then, applying the Codazzi equation \eqref{eqn:2.15}, we get
\begin{equation}\label{eq:4.24}
\begin{split}
(\nabla_{X_1}A)X_3-(\nabla_{X_3}A)X_1
=\tfrac{1}{6}U&+\tfrac{\sqrt{6}}{3}(a_{11}m-a_{12}n)X_1+
\tfrac{\sqrt{6}}{3}(a_{11}n+a_{12}m)X_2\\
&+\tfrac{\sqrt{6}}{3}(a_{11}n+a_{12}m)X_3
+\tfrac{\sqrt{6}}{3}(-a_{11}m+a_{12}n)X_4,
\end{split}
\end{equation}
\begin{equation}\label{eq:4.25}
\begin{split}
(\nabla_{X_1}A)X_4-(\nabla_{X_4}A)X_1=&\tfrac{\sqrt{6}}{3}(a_{11}n+a_{12}m)X_1+\tfrac{\sqrt{6}}{3}(-a_{11}m+a_{12}n)X_2\\
&+\tfrac{\sqrt{6}}{3}(-a_{11}m+a_{12}n)X_3+\tfrac{\sqrt{6}}{3}(-a_{11}n-a_{12}m)X_4.
\end{split}
\end{equation}

Let $\nabla_{X_i}X_j=\sum \Gamma_{ij}^{k}X_k$ with
$\Gamma_{ij}^k=-\Gamma_{ik}^j$, $1\le i,j,k\le 5$. Then, from
\eqref{eq:4.24} and \eqref{eq:4.25}, after calculating the left hand
sides of \eqref{eq:4.24} and \eqref{eq:4.25}, we get
\begin{equation}\label{eq:4.26}
\left\{
\begin{aligned}
&\Gamma_{13}^1=-\sqrt{2}(a_{11}m-a_{12}n),\ \ \Gamma_{13}^2=-\sqrt{2}(a_{11}n+a_{12}m),\\
&\Gamma_{14}^1=-\sqrt{2}(a_{11}n+a_{12}m),\ \
\Gamma_{14}^2=-\sqrt{2}(-a_{11}m+a_{12}n).
\end{aligned}\right.
\end{equation}

Next, \eqref{eq:4.8} gives that
$g(G(X_1,X_2),\xi)=\tfrac{\sqrt{3}}{3}$, and so that
$g(G(X_1,X_2),U)=0$ from \eqref{eq:4.11}. Then by the relations
\eqref{eq:2.3}--\eqref{eq:2.5} we can easily solve
$G(X_1,\xi)=-\tfrac{\sqrt{3}}{3}X_2$. Thus, by the Gauss and
Weingarten formulas, a direct calculation gives that
\begin{equation}\label{eq:4.27}
G(X_1,\xi)=(\tilde \nabla_{X_1} J)\xi=-\sum_{i=1}^5
\Gamma_{15}^{i}X_i+\tfrac{\sqrt{3}}{6}X_3.
\end{equation}
Hence, we have
\begin{equation}\label{eq:4.28}
\Gamma_{15}^{2}=\tfrac{\sqrt{3}}{3},\,\
\Gamma_{15}^{3}=\tfrac{\sqrt{3}}{6},\,\
\Gamma_{15}^{1}=\Gamma_{15}^{4}=0.
\end{equation}

By \eqref{eq:4.26} and \eqref{eq:4.28}, we obtain
\begin{equation}\label{eq:4.29}
\left\{
\begin{aligned}
&\nabla_{X_1}U=\tfrac{\sqrt{3}}{3}X_2+\tfrac{\sqrt{3}}{6}X_3,\\
&\nabla_{X_1}X_1=\Gamma_{11}^2X_2+\sqrt{2}(a_{11}m-a_{12}n)X_3+\sqrt{2}(a_{11}n+a_{12}m)X_4,\\
&\nabla_{X_1}X_2=\Gamma_{12}^1X_1+\sqrt{2}(a_{11}n+a_{12}m)X_3+\sqrt{2}(-a_{11}m+a_{12}n)X_4-\tfrac{\sqrt{3}}{3}U,\\
&\nabla_{X_1}X_3=-\sqrt{2}(a_{11}m-a_{12}n)X_1-\sqrt{2}(a_{11}n+a_{12}m)X_2+\Gamma_{13}^4X_4-\tfrac{\sqrt{3}}{6}U,\\
&\nabla_{X_1}X_4=-\sqrt{2}(a_{11}n+a_{12}m)X_1-\sqrt{2}(-a_{11}m+a_{12}n)X_2+\Gamma_{14}^3X_3.
\end{aligned}\right.
\end{equation}

Now, using that $G(X_1,X_2)=\tfrac{\sqrt{3}}{3}\xi$ and $G(X_1,\xi)=-\tfrac{\sqrt{3}}{3}X_2$, $a_{11}^2+a_{12}^2=\tfrac{1}{2}$ and
$m^{2}+n^{2}=\frac{1}{6}$, \eqref{eqn:4.23} and \eqref{eq:4.29}, by
direct calculations of both sides of
$$
2(\tilde{\nabla}_{X_1} P)X_2=JG(X_1,PX_2) +JPG(X_1,X_2),
$$
we obtain the following equations:
\begin{equation}\label{eq:4.30}
2X_1(a_{12})+2\sqrt{2}m-2a_{11}\Gamma_{12}^1=0,
\end{equation}
\begin{equation}\label{eq:4.31}
-2X_1(a_{11})-2\sqrt{2}n-2a_{12}\Gamma_{12}^1=0,
\end{equation}
\begin{equation}\label{eq:4.32}
\sqrt{6}X_1(m)+2\sqrt{6}n\Gamma_{12}^1=0,
\end{equation}
\begin{equation}\label{eq:4.33}
-\tfrac{4\sqrt{3}}{3}a_{11}+\sqrt{6}X_1(n)-2\sqrt{6}m\Gamma_{12}^1=0.
\end{equation}
Then, carrying the computations $\eqref{eq:4.30}\times
a_{12}-\eqref{eq:4.31}\times a_{11}$ and
$\eqref{eq:4.32}\times m+\eqref{eq:4.33}\times n$,
respectively, we get
$$
a_{11}n=0,\ \ a_{12}m=0.
$$

If $a_{11}=0$, we get $a_{12}^2=\tfrac{1}{2}$, $m=0$ and
$n^2=\tfrac{1}{6}$. Inserting these into \eqref{eq:4.32}, we
obtain $\Gamma_{12}^1=0$. Then by \eqref{eq:4.31}, we have $n=0$.
This yields a contradiction.

If $a_{11}\neq 0$, it holds that $a_{11}^2=\tfrac{1}{2}$,
$a_{12}=0$, $m^2=\tfrac{1}{6}$ and $n=0$. Then by \eqref{eq:4.30}
and \eqref{eq:4.33}, we have
$\tfrac{\sqrt{2}m}{a_{11}}=\Gamma_{12}^1=-\tfrac{\sqrt{2}a_{11}}{3m}$.
This contradicts to the facts $a_{11}^2=\tfrac{1}{2}$ and
$m^2=\tfrac{1}{6}$.

Thus, {\bf Case III} does not occur.

\vskip 1mm

\noindent{\bf Case IV}. $\nu=2$.

\vskip 1mm

In this case, we restrict the discussion on a connected component of
$\Omega_2$. It is easily seen that we are sufficient, without loss
of generality, to consider the following two subcases:

{\bf IV-(i)}: $\lambda=\beta>0,\ \mu\in\{\lambda,-\lambda\}$.

{\bf IV-(ii)}: $\lambda=\beta=0,\ \mu\neq0$.

\vskip 1mm

Actually, for both of the above two subcases, following similar
arguments as in the discussion of Case I from \eqref{eq:4.2} up to
\eqref{eq:4.11}, we can also get $\mu=0$. This is a
contradiction, showing that {\bf Case IV} does not occur.

We have completed the proof of Lemma \ref{lemma:4.1}.
\end{proof}

Next, we have the following Lemma.

\begin{lemma}\label{lemma:4.2}
The case $\dim\mathfrak{D}=2$ does not occur either.
\end{lemma}
\begin{proof}
In this case, we denote still by $\nu,\ \nu\leq5$, the maximum
number of distinct principal curvatures of $M$. Then the set
$M_\nu=\{x\in M\,|\,M$ has exactly $\nu$ distinct principal
curvatures at $x\}$ is a non-empty open subset of $M$. By the
continuity of the principal curvature function, each connected
component of $M_\nu$ is an open subset, the multiplicities of
distinct principal curvatures remain unchanged on each connected
component of $M_\nu$. So we can choose a local smooth frame field
with respect to the principal curvatures.

Now, by assumption $A\phi+\phi A=0$ and Lemma \ref{lemma:2.1}, we
can write \eqref{eq:2.16} as:
\begin{equation}\label{eq:4.34}
\begin{split}
\tfrac{1}{6}g(\phi X,Y)=g((\mu I-A)&G(X,\xi),Y)+g(G((\mu I-A)X,\xi),Y)\\
&-2g(\phi A^{2}X,Y),\ \ X,Y\in \{U\}^{\bot}.
\end{split}
\end{equation}

In a connected component of $M_\nu$, we take a local orthonormal
frame field $\{X_{i}\}_{i=1}^{5}$ of $M$ such that
\begin{equation*}
AX_1=\lambda X_1,\ \ AX_2=\beta X_2,\ \ AX_3=-\lambda X_3, \ \
AX_4=-\beta X_4,\ \ AX_5=\mu X_5,
\end{equation*}
where $X_{3}=JX_1,\, X_4=JX_2,\, X_5=U$. Then, taking
$(X,Y)=(X_1,\phi X_1)$ in \eqref{eq:4.34}, with using $AX_1=\lambda
X_1$ and $A\phi X_1=-\lambda \phi X_1$, we get
$-\tfrac{1}{6}=2\lambda ^{2}$, this is impossible and hence, we have
proved Lemma \ref{lemma:4.2}.
\end{proof}

Finally, from Lemmas \ref{lemma:4.1}, \ref{lemma:4.2} and the fact
that $\dim\,\mathfrak{D}$ can only be $2$ or $4$ at each point of
$M$, we get immediately the assertion of Theorem \ref{thm:1.3}. \qed


\vskip 5mm

\noindent{\sc School of Mathematics and Statistics, Zhengzhou
University, Zhengzhou 450001, People's Republic of China.}

\noindent E-mails: huzj@zzu.edu.cn; yaozkleon@163.com;
zhangxisq@163.com.

\end{document}